\documentclass[runningheads]{llncs}
\usepackage{graphicx}
\def\fC{\mathfrak{C}}

\usepackage{mathtools}
\usepackage{xcolor}
\usepackage{amssymb}
\usepackage{amsmath}

\def\e{\varepsilon}

\begin{document}

\title{Frobenius statistical manifolds $\&$ geometric invariants\\}
%

%
\author{Noemie Combe \and
Philippe Combe\and Hanna Nencka\thanks{This research was supported by the Max Planck Society's Minerva grant. The authors express their gratitude towards MPI MiS for excellent working conditions.}}
\authorrunning{N. Combe, P. Combe, H. Nencka}
%

\institute{Max Planck Institute for Maths in the Sciences, Inselstrasse 22 ,04103 Leipzig, Germany\\ noemie.combe@mis.mpg.de}
%
\vskip-1cm
\maketitle              
\vskip-1cm

\begin{abstract}
In this paper, we explicitly prove that statistical manifolds, related to exponential families and with flat structure connection have a Frobenius manifold structure. This latter object, at the interplay of beautiful interactions between topology and quantum field theory, raises natural questions, concerning the existence of Gromov--Witten invariants for those statistical manifolds. We prove that an analog of Gromov--Witten invariants for those statistical manifolds (GWS) exists. Similarly to its original version, these new invariants have a geometric interpretation concerning intersection points of para-holomorphic curves. However, it also plays an important role in the learning process, since it determines whether a system has succeeded in learning or failed. 
\end{abstract}

\vskip-.2cm
\section{Introduction}
\vskip-.2cm
{\small For more than 60 years statistical manifolds have been a domain of great interest in information theory\, \cite{Ba19}, machine learning~\cite{Am85,Jost,BCN99,Ce3} and in decision theory~\cite{Ce3}. 

Statistical manifolds (related to  exponential families) have an $F$-manifold structure, as was proved in \cite{CoMa}. The notion of $F$-manifolds,  developed in \cite{HeMa99}, arose in the context of mirror symmetry. It is a version of classical Frobenius manifolds, requiring less axioms. 

In this paper, we restrict our attention to statistical manifolds related to {\bf exponential families} (see \cite{Ce3} p.265 for a definition), {\bf being totally geodesic maximal submanifolds} (see \cite{Ce3} p.182 and see \cite{KoNo} p. 180 for foundations of differential geometry), and which have Markov invariant metrics. The latter condition implies the existence of affine flat structures. Using  a {\it purely algebraic framework}, we determine the necessary condition to have a Frobenius manifold structure. 
It is possible to encapsulate a necessary (and important) property of Frobenius' structure within the Witten--Dijkgraaf--Verlinde--Verlinde (WDVV) highly non-linear PDE system: 
\begin{equation}\label{E:WDVV}\forall a,b,c,d: \underset{ef}{\sum}\varPhi_{abe}g^{ef}\varPhi_{fcd}=(-1)^{a(b+c)}\underset{ef}{\sum}\varPhi_{bce}g^{ef}\varPhi_{fad}.\end{equation}
Notice that the (WDVV) system expresses {\it a flatness condition of the manifold} (i.e. vanishing of the curvature), see \cite{Jost,Man99,Sik}.

However, we investigate the (WDVV) condition from a purely {\it algebraic framework}, giving a totally different insight on this problem, and thus avoiding tedious computations of differential equations. 
{\it Algebraically} the (WDVV) condition is expressed by the associativity and potentiality conditions below (see \cite{Man99}, p. 19-20 for a detailed exposition):
\begin{itemize}
\item[$\ast$] {\bf Associativity:} For any (flat) local tangent fields $u,v,w$, we have: \begin{equation}\label{Asso}t(u,v,w)=g(u\circ v,w)= g(u, v\circ w),\end{equation} where $t$ is a rank 3 tensor, $g$ is a metric and $\circ$ is a multiplication on the tangent sheaf.
\item[$\ast$] {\bf Potentiality:} $t$  admits locally everywhere locally a potential function $\Phi$ such that, for any local tangent fields $\partial_i$ we have \begin{equation}\label{Pot}t(\partial_a,\partial_b,\partial_c)=\partial_a\partial_b\partial_c\Phi.\end{equation}
\end{itemize}

We prove explicitly that statistical manifolds  related to exponential families with a flat structure connection (i.e. $\alpha=\pm1 $, in the notation of Amari, indexing the connection $\overset{\alpha}{\nabla}$, see \cite{Am85}) \underline{strictly obey to the axioms of a {\it Frobenius manifold}}. This algebraic approach to statistical Frobenius manifolds allows us to define an analog of Gromov--Witten invariants: the statistical Gromov--Witten invariants (GWS). This plays an important role in the learning process, since it determines whether a system has succeeded in learning or failed. Also, it has a geometric interpretation (as its original version) concerning the intersection of (para-)holomorphic curves.

In this short note, we do not discuss the Cartan--Koszul bundle approach of the affine and projective connection \cite{Ba19} and which is related to the modern algebraic approach of the Kazan--Moscow school \cite{No58,Ro97,Sh02}. 

 
\section{Statistical manifolds $\&$ Frobenius manifolds}
\vskip-.2cm
We rapidly recall material from the previous part (\cite{CoCoNen}). Let $(\Omega,\mathcal{F},\lambda)$ be a measure space, where $\mathcal{F}$ denotes the $\sigma$-algebra of elements of $\Omega$, and $\lambda$ is a $\sigma$-finite measure. We consider the family of parametric probabilities $\frak{S}$ on the measure space $(\Omega,\mathcal{F})$, absolutely continuous wrt $\lambda$. We denote by $\rho_\theta= \frac{dP_{\theta}}{d\lambda}$, the Radon--Nikodym derivative of $P_{\theta}\in\frak{S}$, wrt to $\lambda$ and denote by $S$ the associated family of probability densities of the parametric probabilities. We limit ourselves to the case where $S$ is a smooth topological manifold. 

\[S= \left\{ \rho_{\theta} \in L^1(\Omega, \lambda),\  \theta = \{\theta_1,\dots \theta_n\}; \ \rho_\theta>0\  \lambda - a.e., \ \int_{\Omega}\rho_\theta d\lambda=1\right\}.\]

This generates the space of probability measures absolutely continuous  with respect to the measure $\lambda$, i.e.
$P_{\theta}(A)=\int_A  \rho_{\theta}d\lambda$ where $A\subset \mathcal{F}$.

We construct its tangent space as follows. Let $u\in L^{2}(\Omega, P_{\theta})$ be a tangent vector to $S$ at the point $\rho_{\theta}$. 
\[T_{\theta}=\left\{ u\in L^{2}(\Omega, P_{\theta}); \mathbb{E}_{P_{\theta}}[u]=0, u=\sum_{i=1}^{d}u^i\partial_{i}\ell_{\theta}         \right\},\] 
where $\mathbb{E}_{P_{\theta}}[u]$ is the expectation value, w.r. to the probability $P_{\theta}$.
  
The tangent space of $S$  is isomorphic to the $n$-dimensional linear space generated by the centred random variables (also known as score vector) $\{ \partial_{i}\ell_\theta\}_{i=1}^n$, where $\ell_{\theta} = \ln \rho_{\theta}.$ 

 In 1945, Rao \cite{Rao45} introduced the Riemannian metric on a statistical manifold, using the Fischer information matrix. The statistical manifold forms a (pseudo)-Riemannian manifold. 

In the basis, where $\{ \partial_{i}\ell_\theta\}, \{i=1,\dots,n\}$ where $\ell_{\theta} = \ln \rho_{\theta},$  the Fisher metric are just the {\it covariance matrix} of the score vector. Citing results of \cite{BCN99} (p89) we can in particular state that:

\[  g_{i,j}(\theta)=\mathbb{E}_{P_{\theta}}[\partial_i\ell_{\theta}\partial_{j}\ell_{\theta}]\]
\[  g^{i,j}(\theta)=\mathbb{E}_{P_{\theta}}[a^{i}_{\theta}a^{j}_{\theta}], \]

where  $\{a^{i}\}$ form a dual basis to $\{\partial_j\ell_{\theta}\}$:
\[a^{i}_{\theta}(\partial_j\ell_{\theta})=\mathbb{E}_{P_{\theta}}[a^{i}_{\theta}\partial_j\ell_{\theta}]=\delta^i_{j}\]
with
\[\mathbb{E}_{P_{\theta}}[a^{i}_{\theta}]=0.\]
 
 \begin{definition}
A Frobenius manifold  is a manifold $M$ endowed with an {\it affine flat structure}\footnote{Here the affine flat structure is equivalently described as complete atlas whose transition functions are affine linear. Since the statistical manifolds are (pseudo)-Riemannian manifolds this condition is fulfilled.}, a compatible metric $g$, and an even symmetric rank 3 tensor $t$. Define a symmetric bilinear multiplication on the tangent bundle:
\[\circ: TM \otimes TM\to TM.\]
$M$ endowed with these structures is called {\it pre-Frobenius.}
\end{definition}
 \begin{definition} A pre-Frobenius manifold is Frobenius if it verifies the associativity and potentiality properties defined in the introduction as equations (\ref{Asso}) and (\ref{Pot}).
\end{definition}

\begin{example}
 Take the particular case of $(M, g,t)$, where $t=\nabla -\nabla^*=0$ and $\nabla\neq0$. So, $\nabla=\nabla^*$. 
In this case, the WDVV condition is not satisfied, and therefore the Frobenius manifold axioms are not fulfilled. Indeed, as stated in \cite{Jost} p.199, the (WDVV) equations holds only if the curvature vanishes. Therefore, if $t=0$ and if the manifold has a non flat curvature i.e. $\nabla\neq0$ then this manifold is not a Frobenius one.
\end{example}

We now discuss the necessary conditions to have a  statistical Frobenius manifold. As was stated above, {\bf our attention throughout this paper is restricted only to exponential families, having Markov invariant metrics.}

Let $(S,g,t)$ be a statistical manifold equipped with the (Fischer--Rao) Riemannian metric $g$ and a 3-covariant tensor field $t$ called the {\it skewness tensor}. It is a covariant tensor of rank 3 which is fully symmetric:
\[t:TS \times TS \times TS\to \mathbb{R},\]
given by 
\[ t|_{\rho_{\theta}}(u,v,w)= \mathbb{E}_{P_{\theta}}[u_{\theta}v_{\theta}w_{\theta}]. \]
In other words, in the score coordinates, we have:  
\[t_{ijk}(\theta)=\mathbb{E}_{P_{\theta}}[\partial_i\ell_{\theta}\partial_j\ell_{\theta}\partial_k\ell_{\theta}].\]
Denote the mixed tensor by
$\overline{t}=t.g^{-1}$. It is bilinear map $\overline{t}:TS \times TS \to TS$, given componentwise by: 

\begin{equation}\label{E:1}
\overline{t}^{k}_{ij}= g^{km} t_{ijm},\end{equation}
where $g^{km}=\mathbb{E}_{P_{\theta}}[a^{k}_{\theta}a^{m}_{\theta}]$.
{\bf NB:} This is written using Einstein's convention. 
\begin{remark}The Einstein convention will be used throughout this paper, whenever needed. \end{remark}
We have:
\[\overline{t}_{ij}^k=\overline{t}|_{\rho_{\theta}}(\partial_i\ell_{\theta},\partial_j\ell_{\theta},a^{k})=\mathbb{E}_{P_{\theta}}[\partial_i\ell_{\theta}\partial_j\ell_{\theta}a^k_{\theta}]. \] 
As for the connection, it is given by: 
\[ \overset{\alpha}{\nabla}_XY=\overset{0}{\nabla}_XY + \frac{\alpha}{2}\overline{t}(X,Y), \quad \alpha\in \mathbb{R}, X,Y\in T_{\rho}S\]
where $\overset{\alpha}{\nabla}_XY$ denotes the $\alpha$-covariant derivative.
\begin{remark}
Whenever we have a pre-Frobenius manifold $(S,g,t)$ we call  the connection $\overset{\alpha}{\nabla}$  the {\it structure connection}. 
\end{remark}

In fact $ \overset{\alpha}{\nabla}$ is the unique torsion free connection satisfying: 

\[ \overset{\alpha}{\nabla}g=\alpha t, \] i.e.

\[ \overset{\alpha}{\nabla}_Xg(Y,Z)=\alpha t(X,Y,Z).\]


\begin{proposition}
The tensor $\overline{t}:TS \times TS \to TS$ allows to define a multiplication $\circ$ on $TS$, such that
for all $u,v, \in T_{\rho_{\theta}}S$, we have: 
\[u\circ v = \overline{t}(u,v).\]
 \end{proposition}

\begin{proof}
By construction, in local coordinates, for any $u,v, \in T_{\rho_{\theta}}S$, we have $u=\partial_i\ell_{\theta}$ and $v=\partial_j\ell_{\theta}$.
In particular, $\partial_i\ell_{\theta}\otimes  \partial_j\ell_{\theta}=\overline{t}^{k}_{ij}\partial_k\ell_{\theta}$, which by calculation turns to be 
$\mathbb{E}_{P_{\theta}}[\partial_i\ell_{\theta}\partial_j\ell_{\theta}a^k_{\theta}].$
\end{proof}

\begin{lemma}\label{L:Ass}
For any local tangent fields $u,v,w\in  T_{\rho_{\theta}}S$ the associativity property holds: 
\[g(u\circ v, w)=g(u, v\circ w).\]
\end{lemma} 

\begin{proof}
Let us start with the left hand side of the equation. Suppose that $u= \partial_i\ell_{\theta}$, $v= \partial_j\ell_{\theta}$, $w= \partial_l\ell_{\theta}$. By previous calculations: $\partial_i\ell_{\theta}\circ  \partial_j\ell_{\theta}= \overline{t}^{k}_{ij}\partial_k\ell_{\theta}$.
Insert this result into $g(u\circ v, w)$, which gives us $g(\partial_i\ell_{\theta}\circ  \partial_j\ell_{\theta}, \partial_l\ell_{\theta}),$
and leads to $g( \overline{t}^{k}_{ij}\partial_k\ell_{\theta}, \partial_l\ell_{\theta})$. By some calculations and formula (\ref{E:1}) it turns out to be equal to $t(u,v,w)$.

Consider the right hand side. Let $g(u,v \circ w)=g(\partial_i\ell_{\theta},\partial_j\ell_{\theta}\circ   \partial_l\ell_{\theta}).$
Mimicking the previous approach, we show that this is equivalent to $g(\partial_i\ell_{\theta} , \overline{t}^{k}_{jl}\partial_l\ell_{\theta})$, which  is equal to $t(u,v,w)$.

\end{proof}



\begin{theorem}
The statistical manifold $(S,g,t)$, related to exponential families, is a Frobenius manifold only for $\alpha=\pm1$.
\end{theorem}
\begin{proof}
The statistical manifold $S$ comes equipped with a Riemannian metric $g$, and a skew symmetric tensor $t$.
We have proved that $(S,g,t)$ is an associative pre-Frobenius manifold (the associativity condition is fulfilled, by Lemma \ref{L:Ass}). It remains to show that it is Frobenius.  
We invoke the theorem 1.5 of \cite{Man99}, p. 20 stating that the triplet $(S,g,t)$ is Frobenius if and only if the {\it structure connection} $\overset{\alpha}{\nabla}$ is flat.
The pencil of connections depending on a parameter $\alpha$ are defined by:
\[ \overset{\alpha}{\nabla}_XY=\overset{0}{\nabla}_XY + \frac{\alpha}{2}(X\circ Y), \quad \alpha\in \mathbb{R}, X,Y\in T_{\rho}S\]
where $\overset{\alpha}{\nabla}_XY$ denotes the $\alpha$-covariant derivative.
By a direct computation, we show that only for $\alpha=\pm1$, the structure connection is flat. Therefore, the conclusion is straightforward.  
\end{proof}
The (WDVV)  PDE  version expresses geometrically  a flatness condition for a given manifold. We establish the following connection. 
\begin{proposition}
For $\alpha=\pm 1$, the (WDVV) PDE system are always (uniquely) integrable over $(S,g,t)$. 
\end{proposition}
\begin{proof}
The WDVV equations are always integrable if and only if the curvature is null. 
In the context of $(S,g,t)$ the curvature tensor of the covariant derivative is null for $\alpha=\pm 1$. Therefore, in this context the WDVV equation is always integrable (uniquely). \end{proof}

\begin{corollary}
The Frobenius manifold $(S,g,t)$, related to exponential families, and indexed by $\alpha=\pm1$ verifies the potentiality condition.\footnote{Another interpretation goes as follows. Since for $\alpha =\pm 1$, $(S,g,t)$ is a Frobenius manifold (i.e. a flat manifold), it satisfies the condition to apply theorem 4.4 in\, \cite{Jost}. There exist potential functions (which are convex functions) $\Psi(\theta)$ and $\Phi(\eta)$ such that the metric tensor is given by :
$g_{ij} = \frac{\partial^2}{\partial\theta^i\partial\theta^j}\Phi$ and  $g^{ij} = \frac{\partial^2}{\partial\eta^i\partial\eta^j}\Psi$, where there exists a pair of dual coordinate systems $(\theta,\eta)$ such that $\theta = \{\theta_1,\dots,\theta_n\}$ is $\alpha$-affine and $\eta = \{\eta_1,...,\eta_n\}$ is a $-\alpha$-affine coordinate system. Convexity refers to local coordinates and not to any metric.} 
\end{corollary}


\section{Statistical Gromov--Witten invariants and learning}


\vskip-.2cm

 We introduce Gromov--Witten invariants for statistical manifolds, in short (GWS). Originally,  those invariants are rational numbers that count (pseudo) holomorphic curves under some conditions on a (symplectic) manifold. 
The (GWS) encode deep geometric aspects of statistical manifolds, (similarly to its original version) concerning the intersection of (para-)holomorphic curves. Also, this plays an important role in the learning process, since it determines whether a system has succeeded in learning or failed. 
 
Let us consider the (formal) Frobenius manifold $(H,g)$. We denote $k$ a (super)commutative $\mathbb{Q}$-algebra. Let $H$ be a $k$-module of finite rank and $g:H\otimes H\to k$  an even symmetric pairing (which is non degenerate). We denote $H^*$ the dual to $H$. 
The structure of the Formal Frobenius manifold on $(H,g)$ is given by an even potential ${\bf \Phi} \in k[[H^*]]$:
\[{\bf \Phi}=\sum_{n\geq 3}\frac{1}{n!}Y_n,\]
where $Y_n\in (H^*)^{\otimes n}$ can also be considered as an even symmetric map $H^{\otimes n} \to k$. This system of {\it Abstract Correlation Functions} in $(H,g)$ is a system of (symmetric, even) polynomials. The Gromov--Witten invariants appear in these multi-linear maps. 

We go back to statistical manifolds. Let us consider the discrete case of the exponential family formula:

\begin{equation}\label{E:2}\sum_{\omega\in \Omega} \exp\{-\sum \beta^jX_j(\omega)\}=
\sum_{\omega\in \Omega}\sum_{m\geq 1}\frac{1}{m!}\left\{ -\sum_{j} \beta^jX_j(\omega)\right\}^{\otimes m}, \end{equation}
where $\beta=(\beta_0,....,\beta_n)\in \mathbb{R}^{n+1}$ is a canonical affine parametrisation, $X_j(\omega)$ are directional co-vectors, belonging to a finite cardinality $n+1$ list $\mathcal{X}_n$ of random variables. These co-vectors represent necessary and sufficient statistics of the exponential family. 
We have $X_0(\omega)\equiv 1$, and $X_1(\omega),\dots ,X_{n}(\omega)$ are linearly independent co-vectors. 
The family in (\ref{E:2}) describes an analytical $n$-dimensional hypersurface in the statistical manifold. It can be uniquely determined by $n+1$ points in general position. 


\begin{definition}
Let $k$ be the field of real numbers. Let $S$ be the statistical manifold. The Gromov--Witten invariants for statistical manifolds (GWS)  are given by the multi-linear maps:
 \[\tilde{Y}_n:S^{\otimes n}\to k.\] 
\end{definition}
One can also write them as follows: \[\tilde{Y}_n\in \left(-\sum_{j}\beta^jX_j(\omega)\right)^{\otimes n}.\]

These invariants appear as part of the potential function $\tilde{\bf \Phi}$ which is a Kullback--Liebler entropy function. 

One can write the relative entropy function:
\begin{equation}\label{E:3}\tilde{\bf \Phi}= ln \sum_{\omega\in\Omega} \exp{(-\sum_{j}\beta^jX_j(\omega))}.\end{equation}

Therefore, we state the following:

\begin{proposition}
The entropy function $\tilde{\bf \Phi}$ of the statistical manifold relies on the (GWS).
\end{proposition}
\begin{proof}
Indeed, since $\tilde{\bf \Phi}$, in formula (\ref{E:3}) relies on the polylinear maps $\tilde{Y}_n\in \left(-\sum_{j}\beta^jX_j(\omega)\right)^{\otimes n}$,
 defining the (GWS), the statement follows. 
\end{proof}

Consider the tangent fiber bundle over $S$, the space of probability distributions, with Lie group $G$. We denote it by $(TS,S,\pi,G,F)$, where $TS$ is the total space of the bundle $\pi: TS\to S$ is a continuous surjective map and $F$ the fiber.  
Recall that for any point $\rho$ on $S$, the tangent space at $\rho$ is isomorphic to the space of bounded, signed measures vanishing on an ideal $I$ of the $\sigma-$algebra. The Lie group $G$ acts (freely and transitively) on the fibers by $f\overset{h}{\mapsto} f+h$, where $h$ is a parallel transport, and $f$ an element of the total space (see \cite{CoMa} for details). 

\begin{remark} 
Consider the (local) fibre bundle $\pi^{-1}(\rho)\cong \{\rho\}\times F$. Then $F$ can be identified to a module over the algebra of paracomplex numbers $\fC$ (see \cite{CoMa,CoCoNen} for details). By a certain change of basis, this rank 2 algebra generated by $\{e_1,e_2 \}$, can always be written as $\langle 1,\e |\, \e^2=1 \rangle$.
\end{remark}
We call a {\it canonical basis} for this paracomplex algebra,  the one given by: $\{e_+,e_- \}$, where $e_{\pm}=\frac{1}{2}(1\pm \e).$
Moreover, any vector $X=\{x^{i}\}$ in the module over the algebra is written as $\{x^{ia}e_{a}\}$, where $a\in\{1,2\}$.

\begin{lemma}
Consider the fiber bundle $(TS,S,\pi,G,F)$.
Consider a path $\gamma$ being a geodesic in $S$. Consider its fiber $F_{\gamma}$. Then, the fiber contains two connected compontents: $(\gamma^+,\gamma^-)$, lying respectively in totally geodesic submanifolds $E^+$ and $E^-$. 
\end{lemma}
\begin{proof}
Consider the fiber above $\gamma$. Since for any point of $S$, its the tangent space is identified to  module over paracomplex numbers. This space is decomposed into a pair of subspaces (i.e. eigenspaces with eigenvalues $\pm \e$) (see \cite{CoCoNen}). 
The geodesic curve in $S$ is a path such that $\gamma=(\gamma^i(t)): t\in [0,1]\to S$. In local coordinates, the fiber budle is given by $\{\gamma^{ia}e_{a}\}$, and $a\in \{1,2\}$. Therefore, the fiber over $\gamma$ has two components $(\gamma^+,\gamma^-)$. Taking the canonical basis for $\{e_1,e_2\}$, implies that $(\gamma^+,\gamma^-)$ lie respectively in the subspaces $E^+$ and $E^-$. These submanifolds are totally geodesic in virtue of  Lemma 3 in \cite{CoCoNen}. \end{proof}

We define a learning process through the Ackley--Hilton--Sejnowski method\, \cite{AHS}, which consists in minimising the Kullback--Leibler divergence. 
By Propositions 2 and 3 in\, \cite{CN97}, we can restate it geometrically, as follows:
\begin{proposition}\label{P:CoNen}
The learning process consists in determining if there exist intersections of the paraholomorphic curve $\gamma^+$ with the orthogonal projection of $\gamma^-$ in the subspace $E^+$. 
\end{proposition} 
In particular, a learning process succeeds whenever the distance between a geodesic $\gamma^+$ and  the projected one in $E^+$ shrinks to become as small as possible. 

More formally, as was depicted in\, \cite{BCN99} (sec. 3) let us denote by $\Upsilon$ the set of (centered) random variables over $(\Omega,\mathcal{F},P_{\theta})$
which admit an expansion in terms of the scores under the following form:
\[\Upsilon_P= \{X\in \mathbb{R}^{\Omega}\, |\, X-\mathbb{E}_P[X]=g^{-1}(\mathbb{E}_P[Xd\ell]), d\ell \}. \]
By direct calculation, one finds that the log-likelihood $\ell= ln\rho$ of the usual (parametric) families of probability distributions belongs to 
$\Upsilon_p$  as well as the difference $\ell -\ell^*$  of log-likelihood of two probabilities of the same family. 
Being given a family of probability distributions such that $\ell \in  \Upsilon_P$ for any $P$, let $\mathcal{U}_P$, let us denote $P^*$ the set
 such that $\ell-\ell^*\in  \Upsilon_p$. Then, for any $P^*\in  \mathcal{U}_p$, we define $K(P,P^*)=\mathbb{E}_P[\ell - \ell^*]$.

\begin{theorem}
Let $(S,g,t)$ be  statistical manifold. Then, the (GWS) determine the learning process. 
\end{theorem}
\begin{proof}
Since $K(P,P^*)=\mathbb{E}_P[\ell - \ell^*]$, this implies that $K(P,P^*)$ is minimised whenever there is a successful learning process. 
The learning process is by definition given by  {\it deformation} of a pair of geodesics, defined respectively in the pair of totally geodesic manifolds $E^+, E^-$. Therefore, the (GWS), which arise in the $\tilde{Y}_n$ in the potential function $\tilde{\bf \Phi}$, which is directly related to the relative entropy function $K(P,P^*)$. Therefore, the (GWS) determine the learning process.
 \end{proof}

Similarly as in the classical (GW) case, the (GWS) count intersection numbers of certain para-holomorphic curves. In fact, we have the following statement:
\begin{corollary}
Let  $(TS,S,\pi,G,F)$ be the fiber bundle above. Then, the (GWS) determine the number of intersection of the projected $\gamma^-$ geodesic onto $E^+$, with the $\gamma^+\subset E^+$ geodesic.
\end{corollary}}





\end{document}